\documentclass{amsart}


\usepackage{amsmath,amssymb, amstext,amsthm,amsfonts}
\usepackage[dvips]{graphics}

\thanks{Partially supported by CNPq, FAPERJ and PRONEX/DS from Brazil.}

\newcommand{\cl}{\operatorname{Cl}}

\newcommand{\diff}{\operatorname{\mathfrak{X}^1(M)}}

\newcommand{\per}{\operatorname{Per}}
\newcommand{\sink}{\operatorname{Sink}}
\newcommand{\sou}{\operatorname{Source}}

\newcommand{\sad}{\operatorname{Saddle}}

\newcommand{\di}{\operatorname{div}}

\newcommand{\dis}{\operatorname{Dis}}

\newcommand{\m}{\operatorname{Leb}}

\newcommand{\al}{\alpha}

\newcommand{\SR}{{\mathcal R}}

\newcommand{\SU}{{\mathcal U}}

\newcommand{\card}{\operatorname{card}}

\newcommand{\ang}{\operatorname{angle}}

\newcommand{\Crit}{\operatorname{Crit}}
\newcommand{\Sing}{\operatorname{Sing}}

\newcommand{\psad}{\operatorname{PSaddle}}

\newtheorem{theorem}{Theorem}[section] 
\newtheorem{lemma}[theorem]{Lemma}     
\newtheorem{corollary}[theorem]{Corollary}

\newtheorem*{ara}{Theorem A}
\newtheorem*{arara}{Theorem B}

\title[Attractors, homoclinic tangencies and singular-hyperbolicity]
 {Existence of attractors, homoclinic tangencies and singular-hyperbolicity for flows} 

\author{A. Arbieto, A. Rojas, B. Santiago}

\address{Instituto de Matem\'atica, Universidade Federal do Rio de Janeiro, P. O. Box 68530, 21945-970 Rio
de Janeiro, Brazil.}
\email{arbieto@im.ufrj.br, bruno\_santiago@im.ufrj.br}

\begin{document}
\maketitle

\begin{abstract}
We prove that every $C^1$ generic three-dimensional flow
has either infinitely many sinks, or, finitely many hyperbolic or singular-hyperbolic attractors
whose basins form a full Lebesgue measure set.
We also prove in the orientable case that
the set of accumulation points of the sinks of a $C^1$ generic three-dimensional flow
has no dominated splitting with respect to the linear Poincar\' e flow.
As a corollary we obtain that every three-dimensional flow can be $C^1$ approximated by
flows with homoclinic tangencies or by singular-Axiom A flows.
These results extend \cite{A}, \cite{ah}, \cite{PS} and solve a conjecture in \cite{mp}.
\end{abstract}

\vspace{10pt}

\section{Introduction} 
\label{intro}

\noindent
Araujo's Theorem \cite{A} asserts
that a $C^1$ generic surface diffeomorphism has either infinitely many sinks (i.e. attracting periodic orbits),
or, finitely many hyperbolic attractors
whose basins form a full Lebesgue measure set.
In the recent paper \cite{ams} the authors were able to extend this result
from surface diffeomorphisms to three-dimensional flows without singularities.
More precisely, they proved that a $C^1$ generic three-dimensional flow without singularities
either has infinitely many sinks, or, finitely many hyperbolic attractors whose basins form a full Lebesgue measure set.
The present paper goes beyond and extend
\cite{ams} to the singular case.
Indeed, we prove that every $C^1$ generic three-dimensional flow
has either infinitely many sinks, or, finitely many hyperbolic or singular-hyperbolic attractors
whose basins form a full Lebesgue measure set.
The arguments used in the proof will imply in the orientable case that
the set of accumulation points of the sinks of
a $C^1$ generic three-dimensional flow has no dominated splitting with respect to the linear Poincar\'e
flow. From this we obtain that every three-dimensional flow
can be $C^1$ approximated by
flows with homoclinic tangencies or by singular-Axiom A flows.
This last result extends \cite{ah}, \cite{PS} and solves a conjecture in \cite{mp}.
Let us state our results in a precise way.

By a {\em three-dimensional flow}
we mean a $C^1$ vector fields on compact connected boundaryless manifolds $M$ of dimension $3$.
The corresponding space equipped with the $C^1$ vector field topology will be denoted by $\diff$.
The flow of $X\in\diff$ is denoted by $X_t$, $t\in\mathbb{R}$.
A subset of $\diff$ is {\em residual} if it is a countable intersection of open and dense subsets.
We say that a {\em $C^1$ generic three-dimensional flow satisfies a certain property P} if
there is a residual subset $\mathcal{R}$ of $\diff$ such that P holds for every element of $\mathcal{R}$.
The closure operation is denoted by $\cl(\cdot)$.

By a {\em critical point} of $X$ we mean a point $x$ which is either
{\em periodic} (i.e. there is a minimal $t_{x,X}>0$ satisfying
$X_{t_{x,X}}(x)=x$) or {\em singular} (i.e. $X(x)=0$).
The {\em eigenvalues} of a critical point $x$ are defined respectively as
those of the linear automorphism $DX_{t_{x,X}}(x): T_xM \to T_xM$ not corresponding to
$X(x)$, or, those of $DX(x)$.
A critical point is a {\em sink} if
its eigenvalues are less than $1$ in modulus (periodic case)
or with negative real part (singular case).
A {\em source} will be a sink for the time reversed flow $-X$.
Denote by $\sink(X)$ and $\sou(X)$ the set of sinks and sources of $X$ respectively.

Given a point $x$ we define the {\em omega-limit set},
$$
\omega(x)=\left\{y\in M : y=\lim_{t_k\to\infty}X_{t_k}(x)\mbox{ for some integer sequence }t_k\to\infty\right\}.
$$
(when necessary we shall write $\omega_X(x)$ to indicate the dependence on $X$.)
We call a subset $\Lambda\subset M$
{\em invariant} if $X_t(\Lambda)=\Lambda$ for all $t\in\mathbb{R}$; and
{\em transitive} if there is $x\in\Lambda$ such that
$\Lambda=\omega(x)$.
The {\em basin} of any subset $\Lambda\subset M$ is defined by
$$
W^s(\Lambda)=\{y\in M : \omega(y)\subset \Lambda\}.
$$
(Sometimes we write $W^s_X(\Lambda)$ to indicate dependence on $X$).
An {\em attractor} is a transitive set $A$ exhibiting a neighborhood $U$ such that
$$
A=\displaystyle\bigcap_{t\geq0}X_t(U).
$$
A compact invariant set $\Lambda$ is {\em hyperbolic} if
there are a continuous $DX_t$-invariant tangent bundle decomposition
$T_\Lambda M=E_\Lambda^s\oplus E^X_\Lambda\oplus E_\Lambda^u$ over $\Lambda$
and positive numbers $K,\lambda$ such that $E^X_x$ is generated by $X(x)$,
$$
\|DX_t(x)/E_x^s\|\leq Ke^{-\lambda t}
\quad\mbox{ and } \quad
\|DX_{-t}(x)/E^u_{X_t(x)}\|\leq K^{-1}e^{\lambda t},
\quad\forall (x,t)\in \Lambda\times\mathbb{R}^+.
$$
On the other hand, a {\em dominated splitting} $E\oplus F$ for $X$ over an invariant set $I$
is a continuous tangent bundle $DX_t$-invariant splitting $T_I M=E_I\oplus F_I$
for which there are positive constants $K,\lambda$ satisfying
$$
\|DX_t(x)/E_x\|\cdot\|DX_{-t}(X_t(x))/F_{X_t(x)}\|\leq Ke^{-\lambda t},
\quad\quad\forall (x,t)\in I\times\mathbb{R}^+.
$$
In this case we say that the dominating subbundle $E_I$ is {\em contracting} if
$$
\|DX_t(x)/E_x\|\leq Ke^{-\lambda t},
\quad\quad\forall (x,t)\in I\times\mathbb{R}^+
$$
The central subbundle $F_I$ is said to be {\em volume expanding} if
$$
|\det DX_t(x)/F_x|^{-1}\leq Ke^{-\lambda t},
\quad\quad\forall (x,t)\in I\times\mathbb{R}^+.
$$
A compact invariat set is {\em partially hyperbolic} if it has a dominated splitting with contracting
dominating direction. We say that a partially hyperbolic set is {\em singular-hyperbolic for $X$}
if its singularities are all hyperbolic and its central subbundle is volume expanding.
A {\em hyperbolic} (resp. {\em singular-hyperbolic}) {\em attractor} for $X$ is an attractor which is simultaneously a hyperbolic (resp. singular-hyperbolic) set for $X$.

With these definitions we can state our first result.

\begin{ara}
A $C^1$ generic three-dimensional flow has either infinitely many sinks, or,
finitely many hyperbolic or singular-hyperbolic attractors whose basins form a full Lebesgue measure set.
\end{ara}

The method of the proof of the above result (based on \cite{mpp})
will imply the following result
for three-dimensional flows
on orientable manifolds.
Denote by $\Sing(X)$ the set of singularities of $X$.
Given $\Lambda\subset M$ we denote $\Lambda^*=\Lambda\setminus \Sing(X)$.

We define the vector bundle $N^X$ over $M^*$ whose fiber at $x\in M^*$ is the
the orthogonal complement of $X(x)$ in $T_xM$.
Denoting the projection $\pi_x: T_xM\to N_x^X$ we define the {\em Linear Poincar\'e flow} (LPF)
$P^X_t: N^X\to N^X$ by
$P^X_t(x)=\pi_{X_t(x)}\circ DX_t(x)$, $t\in\mathbb{R}$.
An invariant set $\Lambda$ of $X$
{\em has a LPF-dominated splitting} if $\Lambda^*\neq\emptyset$ and
there exist a continuous tangent bundle decomposition
$N_{\Lambda^*}^X=N^{s,X}_{\Lambda^*}\oplus N^{u,X}_{\Lambda^*}$
with $dim N^{s,X}_x=dim N^{u,X}_x=1$ ($\forall x\in\Lambda^*$) and $T>0$ such that 
$$
\left\|P^X_{T}(x)/N^{s,X}_x\right\|\left\|P^X_{-T}(X_T(x))/N^{u,X}_{X_T(x)}\right\|\leq\frac{1}{2},
\quad\quad\forall x\in\Lambda^*.
$$

\begin{arara}
If $X$ is a $C^1$ generic three-dimensional flow of a orientable manifold, then
neither $\cl(\sink(X))\setminus \sink(X)$ nor $\cl(\sou(X))\setminus \sou(X)$ have LPF-dominated splitting.
\end{arara}

As an application we obtain a solution for Conjecture 1.3 in \cite{mp}.
A periodic point $x$ of $X$ is a {\em saddle} if it has eigenvalues of modulus less and bigger than $1$ simultaneously.
Denote by $\psad(X)$ the set of periodic saddles of $X$.
As is well known \cite{hps}, through any $x\in\psad(X)$ it passes a pair of invariant manifolds,
the so-called strong stable and unstable manifolds $W^{ss}(x)$ and $W^{uu}(x)$,
tangent at $x$ to the eigenspaces corresponding to the eigenvalue of modulus less and bigger than $1$ respectively.
Saturating these manifolds with the flow we obtain the stable and unstable manifolds $W^s(x)$ and $W^u(x)$ respectively.
A {\em homoclinic point} associated to $x$ is a point $q$ where these last manifolds meet.
We say that $q$ is a {\em transverse homoclinic point} if $T_qW^s(x)\cap T_qW^u(x)$ is the one-dimensional subspace generated
by $X(q)$ and a {\em homoclinic tangency} otherwise.

We define the {\em nonwandering set} $\Omega(X)$ as the set of points $p$ such that
for every $T>0$ and every neighborhood $U$ of $p$ there is $t>T$ satisfying
$X_t(U)\cap U\neq\emptyset$.

Following \cite{mp}, we say that $X$ is {\em singular-Axiom A} if
there is a finite disjoint union
$$
\Omega(X)=\Lambda_1\cup\cdots\cup\Lambda_r,
$$
where each $\Lambda_i$ for $1\leq i\leq r$
is a transitive hyperbolic set (if $\Lambda_i\cap \Sing(X)=\emptyset$)
or a singular-hyperbolic attractor for either $X$ or $-X$ (otherwise).

With these definitions we can state the following corollary.

\begin{corollary}
Every three-dimensional flow can be $C^1$ approximated by a flow exhibiting a homoclinic tangency or by a singular-Axiom A flow.
\end{corollary}

\begin{proof}
Passing to a finite covering if necessary we can assume that $M$ is orientable.
Let $R(M)$ denote the set of three-dimensional flows which cannot be $C^1$ approximated by
ones with homoclinic tangencies.
As is well-known \cite{ah}, $\cl(\psad(X))$ has a LPF-dominated splitting for every $C^1$ generic $X\in R(M)$.
Furthermore,
$$
(\cl(\sink(X))\setminus \sink(X))\cup(\cl(\sou(X))\setminus\sou(X))\subset\cl(\psad(X))
$$
Combining this inclusion with Theorem B we obtain
$\cl(\sink(X))\setminus \sink(X)=\cl(\sou(X))\setminus \sou(X)=\emptyset$,
and so, $\sink(X)\cup\sou(X)$ consists of finitely many orbits, for every $C^1$-generic $X\in R(M)$.
Now we obtain that $X$ is singular-Axiom A by Theorem A in \cite{mpp}.
\end{proof}

\section{Proof of theorems A and B}

\noindent
Let $X$ be a three-dimensional flow.
Denote by $\Crit(X)$ the set of critical points.

Recall that a periodic point saddle if it has eigenvalues of modulus less and bigger than $1$ simultaneously.
Analogously for singularities by just replace $1$ by $0$ and the eigenvalues by their corresponding real parts.
Denote by $\sink(X)$ and $\sad(X)$ the set of sinks and saddles of $X$ respectively.

A critical point $x$ is {\em dissipative} if
the product of its eigenvalues (in the periodic case) or the divergence $\di X(x)$ (in the singular case)
is less than $1$ (resp. $0$).
Denote by $\Crit_d(X)$ the set of dissipative critical points. We define
the {\em dissipative region} by $\dis(X)=\cl(\Crit_d(X))$.

For every subset $\Lambda\subset M$ we define
the {\em weak basin} by
$$
W^s_w(\Lambda)=\{x\in M:\omega(x)\cap \Lambda\neq\emptyset\}.
$$
(This is often called
{\em weak region of attraction} \cite{bs}.)
With these notations we obtain the following result. Its proof is similar to the corresponding one in \cite{ams}:

\begin{theorem}
\label{move-attractor}
There is a residual subset $\mathcal{R}_6$ of three-dimensional flows
$X$ for which $W^s_w(\dis(X))$ has full Lebesgue measure.
\end{theorem}

The {\em homoclinic class} associated to $x\in\psad(X)$ is the closure of the set of transverse homoclinic points $q$ associated to $x$.
A homoclinic class of $X$ is the homoclinic class associated to some saddle of $X$.

Given a homoclinic class $H=H_X(p)$ of a three-dimensional flow $X$ we denote by
$H_Y=H_Y(p_Y)$ the continuation of $H$, where $p_Y$ is the analytic continuation of $p$
for $Y$ close to $X$ (c.f. \cite{pt}).

The following lemma was also proved in \cite{ams}. In its statement $\m$ denotes the normalized Lebesgue measure of $M$.

\begin{lemma}
\label{local}
There is a residual subset $\mathcal{R}_{12}$ of three-dimensional flows $X$
such that for every hyperbolic homoclinic class $H$
there are an open neighborhood $\mathcal{O}_{X,H}$ of $f$ and a residual subset
$\mathcal{R}_{X,H}$ of $\mathcal{O}_{X,H}$ such that the following properties are equivalent:
\begin{enumerate}
\item
$\m(W^s_Y(H_Y))=0$ for every $Y\in\mathcal{R}_{X,H}$.
\item
$H$ is not an attractor.
\end{enumerate}
\end{lemma}

We say that a compact invariant set $\Lambda$ of a three-dimensionmal flow $X$ {\em has a spectral decomposition}
if there is a disjoint decomposition
$$
\Lambda=\displaystyle\bigcup_{i=1}^rH_i
$$
into finitely many disjoint homoclinic classes $H_i$, $1\leq i\leq r$,
each one being either hyperbolic (if $H_i\cap\Sing(X)=\emptyset$)
or a singular-hyperbolic attractor for either $X$ or $-X$ (otherwise).

Now we prove the following result which is similar to one in \cite{ams}
(we include its proof for the sake of completeness). In its statement
$\psad_d(X)$ denotes the set of periodic dissipative saddles of a three-dimensional flow $X$.

\begin{theorem}
\label{fui}
There is a residual subset $\mathcal{R}_{11}$ of three-dimensional flows $Y$ 
such that if $\cl(\psad_d(Y))$ has a spectral decomposition, then the following properties are equivalent for
every homoclinic $H$ associated to a dissipative periodic saddle:
\begin{enumerate}
\item[(a)] $\m(W^s_Y(H))>0$.
\item[(b)] $H$ is either hyperbolic attractor or a singular-hyperbolic attractor for $Y$.
\end{enumerate}
\end{theorem}

\begin{proof}
Let  $\mathcal{R}_{12}$ be as in Lemma \ref{local}.
Define the map $S:\diff\to 2^M_c$ by
$S(X)=\cl(\psad_d(X))$.
This map is clearly lower-semicontinuous, and so, upper semicontinuous in a residual subset
$\mathcal{N}$ (for the corresponding definitions see \cite{k1}, \cite{k}).

By Lemma \ref{AO} there is a residual subset $\mathcal{L}$ of three-dimensional flows
$X$ for which every singular-hyperbolic attractor with singularities of either $X$ or $-X$ has zero Lebesgue measure.

By the flow-version of the main result in \cite{a1}, there is a residual subset $\mathcal{R}_7$ of three-dimensional flows
$X$ such that for every singular-hyperbolic attractor $C$ for $X$ (resp. $-X$) there are neighborhoods
$U_{X,C}$ of $C$, $\mathcal{U}_{X,C}$ of $X$ and a residual subset $\mathcal{R}^0_{X,C}$ of $\mathcal{U}_{X,C}$ such that
for all $Y\in\mathcal{R}^0_{X,C}$ if $Z=Y$ (resp. $Z=-Y$) then

\begin{equation}
\label{considera}
C_Y=\displaystyle\bigcap_{t\geq0}Z_t(U_{X,C})
\mbox{ is a singular-hyperbolic attractor for }Z.
\end{equation}

Define $\mathcal{R}= \mathcal{R}_{12}\cap \mathcal{N}\cap\mathcal{L}\cap\mathcal{R}_7$.
Clearly $\mathcal{R}$ is a residual subset of three-dimensional flows.
Define
$$
\mathcal{A}=\{f\in \mathcal{R}:\cl(\psad_d(X))\mbox{ has no spectral decomposition}\}.
$$

Fix $X\in\mathcal{R}\setminus \mathcal{A}$.
Then, $X\in\mathcal{R}$ and
$\cl(\psad_d(X))$ has a spectral decomposition
$$
\cl(\psad_d(X))=\left(\displaystyle\bigcup_{i=1}^{r_X}H^i\right)\cup\left(\displaystyle\bigcup_{j=1}^{a_X} A^j\right)
\cup\left(\displaystyle\bigcup_{k=1}^{b_X} R^k\right)
$$
into hyperbolic homoclinic classes $H_i$
($1\leq i\leq r_X$), singular-hyperbolic attractors $A^j$ for $X$ ($1\leq j\leq a_X$), and
singular-hyperbolic attractors $R^k$ for $-X$ ($1\leq k\leq b_X$).

As $X\in \mathcal{R}_{12}\cap \mathcal{R}_7$, we can consider for each $1\leq i\leq r_X$, $1\leq j\leq a_X$ and $1\leq k\leq b_X$ the neighborhoods
$\mathcal{O}_{X, H^i}$, $\mathcal{U}_{X,A^j}$ and $\mathcal{U}_{X,R^k}$ of $X$ as well as their residual subsets $\mathcal{R}_{X,H^i}$,
$\mathcal{R}^0_{X,A^j}$ and $\mathcal{R}^0_{X,R^k}$ given by Lemma \ref{local} and (\ref{considera}) respectively.

Define
$$
\mathcal{O}_X=\left(\displaystyle\bigcap_{i=1}^{r_X}\mathcal{O}_{X,H^i}\right)\cap
\left(\displaystyle\bigcap_{j=1}^{a_X}\mathcal{U}_{X,A^j}\right)\cap \left(\displaystyle\bigcap_{k=1}^{b_X}\mathcal{U}_{X,R^k}\right)
$$
and
$$
\mathcal{R}_X=\left(\displaystyle\bigcap_{i=1}^{r_X}\mathcal{R}_{X,H^i}\right)\cap
\left(\displaystyle\bigcap_{j=1}^{a_X}\mathcal{R}_{X,A^j}^0\right)\cap \left(\displaystyle\bigcap_{k=1}^{b_X}\mathcal{R}_{X,R^k}^0\right).
$$
Clearly $\mathcal{R}_X$ is residual in $\mathcal{O}_X$.

From the proof of Lemma \ref{local} in \cite{ams} we obtain for each $1\leq i\leq r_X$
a compact neighborhood $U_{X,i}$ of $H^i$ such that
\begin{equation}
\label{toloo}
H^i_Y=\displaystyle\bigcap_{t\in\mathcal{R}}Y_t(U_{X,i})\quad\mbox{ is hyperbolic and equivalent to }H^i,
\quad\quad\forall Y\in \mathcal{O}_{Y,H^i}.
\end{equation}
As $X\in\mathcal{N}$, $S$ is upper semicontinuous at $X$ so we can further assume that
$$
\cl(\psad_d(Y))\subset \left(\displaystyle\bigcup_{i=1}^{r_X} U_{X,i}\right)\cup
\left(\displaystyle\bigcup_{j=1}^{a_X} U_{X,A^j}\right)\cup \left(\displaystyle\bigcup_{k=1}^{b_X} U_{X,R^k}\right),
\quad\quad\forall Y\in\mathcal{O}_{X}.
$$
It follows that
\begin{equation}
\label{tolo}
\cl(\psad_d(Y))=\left(\displaystyle\bigcup_{i=1}^{r_X}H^i_Y\right)\cup
\left(\displaystyle\bigcup_{j=1}^{a_X}A^j_Y\right)\cup \left(\displaystyle\bigcup_{k=1}^{b_X}R^k_Y\right),
\quad\quad\forall Y\in\mathcal{R}_X.
\end{equation}

Next we
take a sequence $X^i\in\mathcal{R}\setminus \mathcal{A}$ which is dense
in $\mathcal{R}\setminus \mathcal{A}$.

Replacing $\mathcal{O}_{X^i}$ by
$\mathcal{O}'_{X^i}$ where
$$
\mathcal{O}'_{X^0}=\mathcal{O}_{X^0}
\mbox{ and }
\mathcal{O}'_{X^i}=\mathcal{O}_{X^i}\setminus\left(\displaystyle\bigcup_{j=0}^{i-1}\mathcal{O}_{X^j}\right), \mbox{ for } i\geq1,
$$
we can assume that the collection
$\{\mathcal{O}_{X^i}:i\in\mathbb{N}\}$ is pairwise disjoint.

Define
$$
\mathcal{O}_{12}=\displaystyle\bigcup_{i\in\mathbb{N}}\mathcal{O}_{X^i}
\quad\mbox{ and }
\quad
\mathcal{R}'_{12}=\displaystyle\bigcup_{i\in\mathbb{N}}\mathcal{R}_{X^i}.
$$

We claim that $\mathcal{R}'_{12}$ is residual in $\mathcal{O}_{12}$.

Indeed, for all $i\in\mathbb{N}$ write $\mathcal{R}_{X^i}=\displaystyle\bigcap_{n\in \mathbb{N}} \mathcal{O}^n_i$,
where $\mathcal{O}^n_i$ is open-dense in
$\mathcal{O}_{X^i}$ for every $n\in\mathbb{N}$.
Since
$\{\mathcal{O}_{X^i}:i\in\mathbb{N}\}$ is pairwise disjoint, we obtain
$$
\displaystyle\bigcap_{n\in\mathbb{N}}
\displaystyle\bigcup_{i\in\mathbb{N}}\mathcal{O}^n_{i}
\subset
\displaystyle\bigcup_{i\in\mathbb{N}}
\displaystyle\bigcap_{n\in\mathbb{N}}\mathcal{O}^n_{i}=\displaystyle\bigcup_{i\in\mathbb{N}}\mathcal{R}_{X^i}=\mathcal{R}'_{12}.
$$
As $\displaystyle\bigcup_{i\in\mathbb{N}}\mathcal{O}^n_{X^i}$
is open-dense in $\mathcal{O}_{12}$, $\forall n\in \mathbb{N}$,
we obtain the claim.

Finally we define
$$
\mathcal{R}_{11}=\mathcal{A}\cup \mathcal{R}'_{12}.
$$
Since $\mathcal{R}$ is a residual subset of three-dimensional flows, we conclude as in Proposition 2.6 of \cite{m}
that $\mathcal{R}_{11}$ is also a residual subset of three-dimensional flows.

Take $Y\in\mathcal{R}_{11}$ such that $\cl(\psad_d(Y))$ has a spectral decomposition
and let $H$ be a homoclinic class associated to a dissipative saddle of $Y$.
Then,
$H\subset \cl(\psad_d(Y))$ by Birkhoff-Smale's Theorem \cite{hk}.
Since $\cl(\psad_d(Y))$ has spectral decomposition, we have $Y\notin \mathcal{A}$ so $Y\in \mathcal{R}'_{12}$
thus $Y\in\mathcal{R}_X$ for some $X\in\mathcal{R}\setminus \mathcal{A}$.
As $Y\in\mathcal{R}_X$, (\ref{tolo}) implies
$H=H^i_Y$ for some $1\leq i\leq r_X$ or $H=A^j_Y$ for some $1\leq j\leq a_X$ or $H=R^k_Y$ for some $1\leq k\leq b_X$.

Now, suppose that $\m(W^{s}_Y(H))>0$. Since $Y\in\mathcal{R}_X$, we have $Y\in \mathcal{R}^0_{X,R^k}$ for all $1\leq k\leq b_X$.
As $X\in\mathcal{L}$, and $W^s_Y(R^k_Y)\subset R^k_Y$ for every $1\leq k\leq b_X$, we conclude by Lemma \ref{AO} that $H\neq R^k_Y$ for every $1\leq k\leq b_X$.

If $H=A^j_Y$ for some $1\leq j\leq a_X$ then
$H$ is an attractor and we are done. Otherwise,
$H=H^i_Y$ for some $1\leq i\leq r_X$. As $Y\in\mathcal{R}_X$, we have
$Y\in \mathcal{R}_{X,H^i}$ and, since $f\in \mathcal{R}_{12}$, we conclude from Lemma \ref{local}
that $H^i$ is an attractor. But by (\ref{toloo}) we have that $H^i_Y$ and $H^i$ are equivalent, so, $H^i_Y$ is an attractor too and we are done.
\end{proof}

We shall need the following lemma which was essentially proved in \cite{ao}.

\begin{lemma}
\label{AO}
There is a residual subset $\mathcal{L}$ of three-dimensional flows $X$ for which every singular-hyperbolic
attractor with singularities of either $X$ or $-X$ has zero Lebesgue measure.
\end{lemma}

\begin{proof}
As in \cite{ao}, for any open set $U$ and any three-dimensional vector field $Y$, let $\Lambda_Y(U)=\bigcap_{t\in\mathbb{R}}Y_t(U)$
be the maximal invariant set of $Y$ in $U$.
Define $\mathcal{U}(U)$ as the set of flows $Y$ such that $\Lambda_Y(U)$ is
a singular-hyperbolic set with singularities of $Y$.
It follows that $\mathcal{U}(U)$ is open in $\diff$.

Now define $\mathcal{U}(U)_n$ as the set of $Y\in \mathcal{U}(U)$ such that
$\m(\Lambda_Y(U))<1/n$. It was proved in \cite{ao} that $\mathcal{U}(U)_n$ is open and dense in $\mathcal{U}(U)$.

Define $\SR(U)_n=\SU(U)_n\cup (\mathfrak{X}^1(M)\setminus\cl(\SU(U))$ which is open and dense set in $\mathfrak{X}^1(M)$.
Let $\{U_m\}$ be a countable basis of the topology, and $\{O_m\}$ be the set of finite unions of such $U_m$'s.
Define
$$
\mathcal{L}=\bigcap_m\bigcap_n\SR(O_m)_n.
$$
This is clearly a residual subset of three-dimensional flows. We can assume without loss of generality that
$\mathcal{L}$ is symmetric, i.e., $X\in\mathcal{L}$ if and only if $-X\in \mathcal{L}$.
Take $X\in \mathcal{L}$. Let $\Lambda$ be a singular-hyperbolic attractor for $X$.
Then, there exists $m$ such that $\Lambda=\Lambda_X(O_m)$.
Then $X\in \SU(O_m)$ and so $X\in \SU(O_m)_n$ for every $n$ thus $\m(\Lambda)=0$.
Analogously, since $\mathcal{L}$ is symmetric, we obtain that $\m(\Lambda)=0$ for every singular-hyperbolic attractor with singularities
of $-X$.
\end{proof}

In the sequel we obtain the following key result representing the new ingredient with respect to \cite{ams}.
Its proof will use the methods in \cite{mpp}.
In its statement $\card(\sink(X))$ denotes the cardinality of the set of {\em different} orbits of a three-dimensional flow
$X$ contained in $\sink(X)$.

\begin{theorem}
\label{peo}
There is a residual subset $\mathcal{Q}$ of three-dimensional flows $X$
such that if $\card(\sink(X))<\infty$, then $\cl(\psad_d(X))$ has a spectral decomposition.
\end{theorem}

\begin{proof}
First we state some useful notiations.

Given a three-dimensional flow $Y$ and a point $p$
we denote by $O_Y(p)=\{Y_t(p):t\in\mathbb{R}\}$ the $Y$-orbit of $p$.
If $p\in\psad_d(Y)$ we denote by
$E^{s,Y}_p$ and $E^{u,Y}_p$ the eigenspaces corresponding to the eigenvalues
of modulus less and bigger than $1$
respectively.

Denote by $\lambda(p,Y)$ and $\mu(p,Y)$ the eigenvalues of $p$ satisfying
$$
|\lambda(p,Y)|<1<|\mu(p,Y)|.
$$

Define the {\em index} of a singularity $\sigma$ as the number $Ind(\sigma)$ of
eigenvalues with negative real part.

We say that a singularity $\sigma$ of $Y$ is {\em Lorenz-like for $Y$}
if its eigenvalues $\lambda_1,\lambda_2,\lambda_3$ are real and satisfy
$\lambda_2<\lambda_3<0<-\lambda_3<\lambda_1$ (up to some order).
It follows in particular that $\sigma$ is {\em hyperbolic} (i.e. without eigenvalues of zero real part)
of index $2$.
Furthermore, the invariant manifold theory \cite{hps} implies
the existence of {\em stable} and {\em unstable} manifolds
$W^{s,Y}(\sigma)$, $W^{u,Y}(\sigma)$ tangent at $\sigma$ to the eigenvalues $\{\lambda_2,\lambda_3\}$ and $\lambda_1$ respectively.
There is an additional invariant manifold $W^{ss,Y}(\sigma)$, the {\em strong stable manifold}, contained in $W^{s,Y}(\sigma)$ and tangent at
$\sigma$ to the eigenspace corresponding to $\lambda_1$.
We shall denote by $E^{ss,Y}_\sigma$ and $E^{cu,Y}_\sigma$ the eigenspaces associated to the set of eigenvalues
$\lambda_2$ and $\{\lambda_3,\lambda_1\}$ respectively.

Let $S(M)$ be the set of three-dimensional flows $X$ with $\card(\sink(X))<\infty$ such that
$$
\card(\sink(Y))=\card(\sink(X)), \quad\mbox{ for every }Y\mbox{ close to }X.
$$

\vspace{5pt}

Every $X\in S(M)$ satisfies the following properties:

\begin{itemize}
\item
There is a LPF-dominated splitting over
$\psad^*_d(X)\setminus \Sing(X)$, where $\psad_d^*(X)$ denotes the set of points
$x$ for which there are sequences $Y_k\to X$ and $x_k\in \psad_d(X_k)$ such that
$x_k\to x$ (c.f. \cite{w0}).

\vspace{5pt}

\item
There are a neighborhood $\mathcal{U}_X$, $0<\lambda<1$ and $\alpha>0$ such that
if $(p,Y)\in \psad_d(Y)\times \mathcal{U}_X$, then
\begin{enumerate}
\item[(a)]
\begin{enumerate}
\item[1.]
$|\lambda(p,Y)|<\lambda^{t_{p,Y}}$,
\item[2.]
$|\mu(p,Y)|>\lambda^{-t_{p,Y}}$.
\end{enumerate}
\item[(b)]
$\ang(E^{s,Y}_p,E^{u,Y}_p)>\al$.
\end{enumerate}
\end{itemize}

Indeed, the first property follows from the proof of Proposition 5.3 in \cite{ams} and the second from the proof
of Theorem 3.6 in \cite{mpp} (see also the proof of lemmas 7.2 and 7.3 in \cite{ams}).

\vspace{5pt}

In addition to this we also have the existence of a residual subset of three-dimensional flows
$\mathcal{R}_7$ such that every $X\in S(M)\cap \mathcal{R}_7$ satisfies that:

\begin{itemize}
\item
Every $\sigma\in\Sing(X)\cap\cl(\psad_d(X))$ with $Ind(\sigma)=2$ is Lorenz-like for $X$ and satisfies
$\cl(\psad_d(X))\cap W^{ss,X}(\sigma)=\{\sigma\}$.

\vspace{5pt}

\item
Every $\sigma\in\Sing(X)\cap\cl(\psad_d(X))$ with $Ind(\sigma)=1$ is Lorenz-like for $-X$ and satisfies
$\cl(\psad_d(X))\cap W^{uu,X}(\sigma)=\{\sigma\}$, where
$W^{uu,X}(\sigma)=W^{ss,-X}(\sigma)$.
\end{itemize}

Indeed, as in the remark after Lemma 2.13 in \cite{bgy},
there is a residual subset $\mathcal{R}_7$ of three-dimensional flows $X$ such that
every $\sigma\in\Sing(X)$ accumulated by periodic orbits is Lorenz-like for either $X$ or $-X$ depending on whether $\sigma$ has three
real eigenvalues $\lambda_1,\lambda_2,\lambda_3$ satisfying either
$\lambda_2<\lambda_3<0<\lambda_1$ or
$\lambda_2<0<\lambda_3<\lambda_1$ (up to some order).

Now, take $X\in S(M)\cap \mathcal{R}_7$.
Since $X\in S(M)$, we have that $\psad^*_d(X)\setminus \Sing(X)$ has a LPF-dominated splitting and then $\cl(\psad_d(X))\setminus \Sing(X)$
also does because $\cl(\psad_d(X))\subset \psad_d^*(X)$. Therefore,
if $\sigma\in\Sing_2(X)\cap\cl(\psad_d(X))$, Proposition 2.4 in \cite{d} implies
that $\sigma$ has three different real eigenvalues $\lambda_1,\lambda_2,\lambda_3$
satisfing $\lambda_2<\lambda_3<0<\lambda_1$ (up to some order).
Since $X\in\mathcal{R}_7$, we conclude that $\sigma$ is Lorenz-like for $X$.
To prove $\cl(\psad_d(X))\cap W^{ss,X}(\sigma)=\{\sigma\}$ we assume by contradiction that this is not the case.
Then, there is $x\in (\cl(\psad_d(X))\cap W^{ss,X}(\sigma))\setminus \{\sigma\}$.
Choose sequences $x_n\in\cl(\psad_d(X))$ and $t_n\to\infty$ such that $x_n\to x$
and $X_{t_n}(x_n)\to y$ for some $y\in W^{u,X}(\sigma)\setminus\{\sigma\}$.
Let $N^{s,X}\oplus N^{u,X}$ denote the LPF-dominated splitting of $\cl(\psad_d(X))\setminus \Sing(X)$.
We have $N^{s,X}_x=N_x\cap W^{s,X}(\sigma)$ by Proposition 2.2 in \cite{d} and so
$N_{x_n}$ tends to be tangent to $W^{s,X}(\sigma)$ as $n\to\infty$.
On the other hand, Proposition 2.4 in \cite{d} says that $N^{s,X}_{y}$ is almost parallel to $E^{ss,X}_\sigma$.
Therefore, the directions $N^{s,X}_{X_{t_n}(x_n)}$ tends to have positive angle with $E^{ss,X}_\sigma$.
But using that $\lambda_2<\lambda_3$ we can see that $N^{s,X}_{x_n}=P_{-t_n}(X_{t_n}(x_n))N^{s,X}_{X_{t_n}(x_n)}$ tends to be
transversal to $W^{s,X}(\sigma)$ nearby $x$.
As this is a contradiction, we obtain the result.
The second property can be proved analogously.

\vspace{5pt}

On the other hand,
there is another residual subset $\mathcal{Q}_1$ of three-dimensional flows
for which every compact invariant set without singularities but with a LPF-dominated splitting is hyperbolic.

Indeed, by Lemma 3.1 in \cite{bgy} we have that there is a residual subset $\mathcal{Q}_1$ of three-dimensional flows
for which every transitive set without singularities but with a LPF-dominated splitting is hyperbolic.
Fix $X\in\mathcal{Q}_1$ and a compact invariant set $\Lambda$ without singularities but with a LPF-dominated
splitting $N_\Lambda^X=N^{s,X}_\Lambda\oplus N^{u,X}_\Lambda$.
Suppose by contradiction that $\Lambda$ is not hyperbolic.
Then, by Zorn's Lemma, there is a minimally nonhyperbolic set $\Lambda_0\subset \Lambda$ (c.f. p.983 in \cite{PS}).
Assume for a while that $\Lambda_0$ is not transitive.
Then, $\omega(x)$ and $\alpha(x)=\omega_{-X}(x)$ are proper subsets of $\Lambda_0$, $\forall x\in\Lambda_0$.
Therefore, both sets are hyperbolic and then we have
$$
\lim_{t\to\infty}\|P^X_t(x)/N^{s,X}_x\|=\lim_{t\to \infty}\|P^X_{-t}(x)/N^{u,X}_x\|=0,
\quad\quad\forall
x\in \Lambda_0,
$$
which easily implies that $\Lambda_0$ is hyperbolic.
Since this is a contradiction, we conclude that $\Lambda_0$ is transitive.
As $X\in\mathcal{Q}_1$ and $\Lambda_0$ has a LPF-dominated splitting (by restriction), we conclude that $\Lambda_0$ is
hyperbolic, a contradiction once more proving the result.

\vspace{5pt}

Next we recall that a compact invariant set $\Lambda$ of a flow $X$ is {\em Lyapunov stable for $X$}
if for every neighborhood $U$ of $\Lambda$ there is a neighborhood $V\subset U$ of $\Lambda$ such that
$X_t(V)\subset U$, for all $t\geq 0$.

It follows from \cite{cmp}, \cite{mp} that there is a residual subset $\mathcal{D}$ of three-dimensional flows $X$ such that
if $\sigma\in\Sing(X)\cap\cl(\psad_d(X))$ and $Ind(\sigma)=2$, then
$\cl(W^u(\sigma))$ is a Lyapunov stable set for $X$ with dense singular unstable branches
contained in $\cl(\psad_d(X))$.
Analogously, if
$Ind(\sigma)=1$, then
$\cl(W^s(\sigma))$ is a Lyapunov stable set for $-X$ with dense singular stable branches
contained in $\cl(\psad_d(X))$.

From these properties we derive easily that every $X\in S(M)\cap\mathcal{R}_7\cap \mathcal{D}$ and
every $\sigma\in\Sing(X)\cap\cl(\psad_d(X))$ satisfies one of the following alternatives:

\begin{enumerate}
\item[(c)]
If $Ind(\sigma)=2$, then
every $\sigma'\in\Sing(X)\cap\cl(W^u(\sigma))$ is Lorenz-like for $X$.
\item[(d)]
If $Ind(\sigma)=1$, then
every $\sigma'\in\Sing(X)\cap\cl(W^s(\sigma))$ is Lorenz-like for $-X$.
\end{enumerate}

\vspace{5pt}

Given a three-dimensional flow $Y$
we define
$$
E^{cu,Y}_p=E^{u,Y}_p\oplus E^Y_p, \quad\quad\forall p\in\psad_d(Y).
$$

We claim that
there is a residual subset of three-dimensional flows $\mathcal{R}_{15}$ such that for every $X\in S(M)\cap\mathcal{R}_{15}$
and every $\sigma\in \Sing(X)\cap\cl(\psad_d(X))$ there are neighborhoods $\mathcal{V}_X$ of $X$, $U_\sigma$ of $\sigma$ and
$\beta_\sigma>0$ such that if $Y\in \mathcal{V}_X$ and $x\in \psad_d(Y)$ satisfies $O_Y(x)\cap U_\sigma\neq\emptyset$, then
\begin{equation}
\label{residua}
\ang(E^{s,Y}_x,E^{cu,Y}_x)>\beta_\sigma,\quad\quad\mbox{ if }Ind(\sigma)=2
\end{equation}
and
\begin{equation}
\label{residua'}
\ang(E^{s,-Y}_x,E^{cu,-Y}_x)>\beta_\sigma,\quad\quad\mbox{ if }Ind(\sigma)=1.
\end{equation}

(This step corresponds to Theorem 3.7 in \cite{mpp}.)

\vspace{5pt}

Indeed, we just take
$\mathcal{R}_{15}=\mathcal{Q}_1\cap \mathcal{D}\cap \mathcal{R}_7\cap \mathcal{I}$ where $\mathcal{I}$ is the set of upper semicontinuity points of the
the map $\varphi:X\mapsto \cl(\psad_d(X))$.

To prove (\ref{residua}) it suffices to show the following assertions, correponding to propositions 4.1 and 4.2 of \cite{mpp} respectively,
for any $X\in S(M)\cap \mathcal{R}_{15}$ and $\sigma\in\Sing(X)\cap\cl(\psad_d(X))$ with $Ind(\sigma)=2$
($B_\delta(\cdot)$ denotes the $\delta$-ball operation):

\vspace{5pt}

\begin{enumerate}
\item[{\em A1.}]
Given $\epsilon>0$ there are a neighborhood $\mathcal{V}_{X,\sigma}$ of $X$ and $\delta>0$
such that for all $Y\in\mathcal{V}_{X,\sigma}$ if $p\in\psad_d(Y)\cap B_\delta(\sigma_Y)$ then
\begin{enumerate}
 \item
$\ang(E^{s,Y}_p,E^{ss,Y}_{\sigma_Y})<\epsilon$;
\item$\ang(E^{cu,Y}_p,E^{cu,Y}_{\sigma_Y})<\epsilon$.
\end{enumerate}

\vspace{5pt}

\item[{\em A2.}]
Given $\delta>0$ there are a neighborhoof $\mathcal{O}$ of $X$ and $C>0$ such that if $Y\in\mathcal{O}$ and
$p\in\psad_d(Y)$ with $dist(p,\Sing(X)\cap\cl(\psad_d(X)))>\delta$, then
$$
\ang(E^{s,Y}_p,E^{cu,Y}_p)>C.
$$
\end{enumerate}

To prove A1-(a) we proceed as in p. 417 of \cite{cmp}.
By contradiction suppose that it is not true.
Then, there are $\gamma>0$ and sequences $Y^n\to X$, $p_n\in \psad_d(Y^n)\to \sigma$
such that
$$
\ang(E^{s,Y^n}_{p_n},E^{ss,Y^n}_{\sigma_{Y^n}})>\gamma,
\quad\quad\forall n\in \mathbb{N}.
$$
As in \cite{mpp} we take small cross sections $\Sigma^s_{\delta,\delta'}$ and $\Sigma^u_\delta$ located close to the singularities
in $\cl(W^u(\sigma))$ all of which are Lorenz-like (by (c) above).
It turns out that since $p_n\to \sigma$, there are times $t_n\to\infty$ satisfying $q_n=Y^n_{t_n}(p_n)\in \Sigma^u_\delta$.
Using the above inequality we obtain
$$
\ang(E^{s,Y^n}_{q_n},E^{Y^n}_{q_n})\to 0.
$$
Next  consider the first $s_n>0$ such that
$$
\tilde{q}_n=Y^n_{s_n}(q_n)\in \Sigma^s_{\delta,\delta'}.
$$
We obtain
\begin{equation}
\label{oculo}
\ang(E^{s,Y^n}_{\tilde{q}_n},E^{Y^n}_{\tilde{q}_n})\to0.
\end{equation}
To see why, we assume two cases: either $s_n$ is bounded or not.
If it does, then the above limit follows from the corresponding one for $q_n$.
If not, we consider a limit point $q$ of the sequence $Y^n_{\frac{s_n}{2}}(q_n)$ with $s_n\to\infty$.
After observing that the $X$-orbit of $q$ cannot accumulate any index $1$ singularity we obtain
easily that $q\in \Gamma$, where
$$
\Gamma=\displaystyle\bigcap_{t\in\mathbb{R}}X_t\left(\cl(\psad_d(X))\setminus B_{\delta^*}(\Sing(X)\cap\cl(\psad_d(X)))\right),
$$
for some $\delta^*>0$ small.
Clearly $\Gamma$ is a compact invariant subset of $X$ contained in $\cl(\psad_d(X))\setminus \Sing(X)$.
Since $X\in S(M)$, we have that $\Gamma$ has a LPF-dominated splitting, and so, it is hyperbolic because $X\in\mathcal{Q}_1$.
This allows us to repeat the proof in p. 419 to obtain (\ref{oculo}) which, together with (b) above,
implies that $\ang(E^{u,Y^n}_{\tilde{q_n}},E^{Y^n}_{\tilde{q_n}})$ is bounded away from zero.
But now we consider the first positive time $r_n$ satisfying $\tilde{\tilde{q}}_n=Y^n_{r_n}(\tilde{q}_n)\in \Sigma^u_\delta$.
We get as in p. 419 in \cite{mpp} that $\ang(E^{s,Y^n}_{\tilde{\tilde{q}}_n}, E^{Y^n}_{\tilde{\tilde{q}}_n})\to 0$
and, since $\ang(E^{u,Y^n}_{\tilde{q_n}},E^{Y^n}_{\tilde{q_n}})$ is bounded away from $0$, we also obtain
$\ang(E^{u,Y^n}_{\tilde{\tilde{q}}_n}, E^{Y^n}_{\tilde{\tilde{q}}_n})\to 0$.
All this together yield
$\ang(E^{s,Y^n}_{\tilde{\tilde{q}}_n}, E^{u,Y^n}_{\tilde{\tilde{q}}_n})\to 0$ which contradicts (b).
This contradiction completes the proof of A1-(a).
The bound in A1-(b) follows easily from the methods in \cite{d}.
This completes the proof of A1.
A2 follows exactly as in p. 421 of \cite{mpp}.
Now A1 and A2 imply (\ref{residua}) as in \cite{mpp}.
To prove (\ref{residua'}) we only need to repeat the above proof
with $-Y$ instead of $Y$ taking into account the symmetric relations below:
$$
\lambda(p,-Y)=\mu^{-1}(p,Y),\,
\mu(p,-Y)=\lambda^{-1}(p,Y),\, E^{s,-Y}_p=E^{u,Y}_p\mbox{ and }E^{u,-Y}_p=E^{s,Y}_p.
$$

\vspace{5pt}

Once we prove 
(\ref{residua}) and (\ref{residua'}) we use them together with (a) and (b), as in the proof of Theorem F in \cite{cmp}, to obtain
that for every $X\in\mathcal{R}_{15}\cap S(M)$ there is a neighborhood $\mathcal{K}_X$, $0<\rho<1$, $c>0$, $\delta>0$ and $T_0>0$ satisfying
the following properties for every $Y\in\mathcal{K}_X$ and every $x\in \psad_d(Y)$ satisfying $t_{x,Y}>T_0$ and $O_Y(x)\cap B_\delta(\sigma)\neq\emptyset$:
\begin{itemize}
\item
If $Ind(\sigma)=2$, then
$$
\|DY_T(p)/E^{s,Y}_p\|\cdot \|DY_{-T}(p)/E^{cu,Y}_{Y_{-T}(p)}\|\leq c\rho^T, \quad\quad\forall T>0.
$$
\item
If $Ind(\sigma)=1$, then
$$
\|D(-Y)_T(p)/E^{s,-Y}_p\|\cdot \|D(-Y)_{-T}(p)/E^{cu,-Y}_{(-Y)_{-T}(p)}\|\leq c\rho^T, \quad\quad\forall T>0.
$$
\end{itemize}

Since we can assume that $X$ is Kupka-Smale (by the Kupka-Smale Theorem \cite{hk}), the set of periodic orbits with period $\leq T_0$
of $X$ in $\psad_d(X)$ is finite.
If one of these orbits (say $O$) do not belong to $\cl(\cl(\psad_d(X))\setminus \{x\in\psad_d(X):t_x<T_0\})$
then it must happen that $O$ is isolated in the sense
that $\cl(\psad_d(X))\setminus O$ is a closed subset.
Therefore, up to a finite number of isolated periodic orbits,
we can assume that the set
$\psad_d^{T_0}(X)=\{p\in\psad_d(X):t_{p,X}\geq T_0\}$ is dense in $\cl(\psad_d(X))$.
Then, as in p.400 of \cite{mpp} we obtain the following properties:

\begin{itemize}
\item
If  $Ind(\sigma)=2$, then
the splitting $E^{s,X}\oplus E^{cu,X}$
extends to a dominated splitting $E\oplus F$ for $X$ over $\cl(W^u(\sigma))$ with $dim(E)=1$ and $E^X\subset F$.
\item
If $Ind(\sigma)=1$ the splitting $E^{s,-X}\oplus E^{cu,-X}$ extends to a dominated splitting $E\oplus F$ for $-X$ over $\cl(W^s(\sigma))$
with $dim(E)=1$ and $E^{-X}\subset F$.
\end{itemize}

Therefore, we conclude from (c) and (d) above, lemmas 3.2 and 3.4 in \cite{bgy} and Theorem D in \cite{mp} that
if $X\in\mathcal{R}_{15}\cap S(M)$ and $\sigma\in\Sing(X)\cap \cl(\psad_d(X))$, then:

\begin{itemize}
\item
If $Ind(\sigma)=2$, then $\cl(W^u(\sigma))$ is a singular-hyperbolic attractor for $X$.
\item
If $Ind(\sigma)=1$, then
$\cl(W^s(\sigma))$ is a singular-hyperbolic attractor for $-X$.
\end{itemize}

Next, we define $\phi: \diff\to 2^M_c$ by $\phi(X)=\cl(\sink(X))$.
This map is clearly lower semicontinuous, and so, upper semicontinuous in a residual subset $\mathcal{C}$ of $\diff$ (\cite{k}, \cite{k1}).
If $X\in \mathcal{C}$ satisfies $\card(\sink(X))<\infty$, then the upper semicontinuity of $\phi$ at $X$ do imply
$X\in S(M)$.

Finally we define
$$
\mathcal{Q}=\mathcal{R}_{15}\cap \mathcal{C}.
$$
Clearly $\mathcal{Q}$ is a residual subset of three-dimensional flows.

Take $X\in\mathcal{Q}$ with $\card(\sink(X))<\infty$.
Since $X\in \mathcal{C}$, we obtain $X\in S(M)$ thus $X\in\mathcal{R}_{15}\cap S(M)$.
Then, if $\sigma\in\Sing(X)\cap \cl(\psad_d(X))$,
$\cl(W^u(\sigma))$ is singular-hyperbolic for $X$ (if $Ind(\sigma)=2$)
and that $\cl(W^s(\sigma))$ is a singular-hyperbolic attractor for $-X$ (if $Ind(\sigma)=1$).

Now we observe that if $p\in\psad_d(X)$ then $H(p)\subset \cl(\sad_d(X))$ by the Birkhoff-Smale Theorem. From this we obtain
\begin{equation}
\label{spectral}
\cl(\psad_d(X))=
\cl\left(\displaystyle\bigcup\{H(p):p\in\psad_d(X)\}\right).
\end{equation}

We claim that the family $\{H(p):p\in\psad_d(X)\}$ is finite.
Otherwise, there is an infinite sequence $p_k\in\psad_d(X)$ yielding
infinitely many distinct homoclinic classes $H(p_k)$.
Consider the closure $\cl(\bigcup_k H(p_k))$, which is a compact invariant set contained in $\cl(\psad_d(X))$.
If this closure does not contain any singularity, then it would be a hyperbolic set (this follows because $\mathcal{R}_{15}\subset \mathcal{Q}_1$).
Since there are infinitely many distinct homoclinic classses in this closure, we obtain a contradiction proving that
$\cl(\bigcup_k H(p_k))$ contains a singularity $\sigma\in\cl(\psad_d(X))$.
If $Ind(\sigma)=2$ then $\sigma$ lies in $\cl(W^u(\sigma))$ which is an attractor, and so,
we can assume that $H(p_k)\subset \cl(W^u(\sigma))$ for every $k$ thus
$H(p_k)=\cl(W^u(\sigma))$ for every $k$ which is absurd. Analogously for $Ind(\sigma)=1$
and the claim is proved.
Combining with (\ref{spectral}) we obtain the desired spectral decomposition.
\end{proof}

\vspace{5pt}

\begin{proof}[Proof of Theorem A]
Define $\mathcal{R}=\mathcal{R}_6\cap\mathcal{R}_{11}\cap\mathcal{Q}$, where $\mathcal{R}_6$, $\mathcal{R}_{11}$ and $\mathcal{Q}$
are the residual subsets given by
theorems \ref{move-attractor}, \ref{fui} and \ref{peo} respectively.
Suppose that $X\in\mathcal{R}$ has no infinitely many sinks. Then, $\card(\sink(X))<\infty$.
Since $X\in\mathcal{Q}$,
we conclude by Theorem \ref{peo} 
that $\cl(\psad_d(X))$ has a spectral decomposition. Since $X\in\mathcal{R}_{11}$, Theorem \ref{fui} implies that
every homoclinic $H$ associated to a dissipative periodic saddle of $X$ with
$\m(W^s_Y(H))>0$ is an attractor of $X$.
Since $X\in\mathcal{R}_6$, we have that
$\m(W^s_w(\dis(X)))=1$ by Theorem \ref{move-attractor}.

Now, we consider the following decomposition:
$$
\dis(X)
=\cl(\sad_d(X)\cap\Sing(X))\cup\cl(\psad_d(X))\cup\sink(X),
$$
valid in the Kupka-Smale case (which is generic). From this we obtain the union
$$
W^s_w(\dis(X))=
\left(\bigcup\{W^s(\sigma):\sigma\in\sad_d(X)\cap\Sing(X)\mbox{ and }W^s_w(\sigma)=W^s(\sigma)\}\right)
\cup
$$
$$
\left(\bigcup\{W^s_w(\sigma):\sigma\in\sad_d(X)\cap\Sing(X)\mbox{ and }W^s_w(\sigma)\neq W^s(\sigma)\}\right)\cup
$$
$$
W^s_w(\cl(\psad_d(X)))\cup W^s(\sink(X)).
$$
But it is easy to check that the first element in the right-hand union above has zero Lebesgue measure and, by the Hayashi's connecting lemma \cite{h},
we can assume without loss of generality that every $\sigma\in\sad_d(X)\cap\Sing(X)$ satisfying $W^s_w(\sigma)\neq W^s(\sigma)$ lies in
$\cl(\psad_d(X))$. Since $W^s_w(\dis(X))$ has full Lebesgue measure, we conclude that
$$
\m(W^s_w(\cl(\psad_d(X)))\cup W^s(\sink(X)))=1.
$$
Now, we use the spectral decomposition
$$
\cl(\psad_d(X))=\displaystyle\bigcup_{i=1}^rH_i
$$
into finitely many disjoint homoclinic classes $H_i$, $1\leq i\leq r$,
each one being either hyperbolic (if $H_i\cap\Sing(X)=\emptyset$)
or a singular-hyperbolic attractor for either $X$ or $-X$ (otherwise), yielding
$$
\m\left(\left(\displaystyle\bigcup_{i=1}^rW^s_w(H_i)\right)\cup W^s(\sink(X))\right)=1.
$$
But the results in Section 3 of \cite{cmp}
imply that each $H_i$ can be written as
$H_i=\Lambda^+\cap \Lambda^-$,
where $\Lambda^\pm$ is a Lyapunov stable set for $\pm X$.
We conclude from Lemma 2.2 in \cite{cmp} that $W^s_w(H_i)=W^s(H_i)$ thus
$$
\m\left(\left(\displaystyle\bigcup_{i=1}^rW^s(H_i)\right)\cup W^s(\sink(X))\right)=1.
$$
Let $1\leq i_1\leq \cdots \leq i_d\leq r$ be such that $\m(W^s(H_{i_k}))>0$ for every $1\leq k\leq d$.
As the basin of the remainder homoclinic classes in the collection
$H_1,\cdots,H_r$ are negligible, we can remove them from the above union yielding
$$
\m\left(
\left(\displaystyle\bigcup_{k=1}^dW^s(H_{i_k})\right)\cup\left(\displaystyle\bigcup_{j=1}^lW^s(s_j)\right)
\right)=1,
$$
where the $s_j$'s above correspond to the finitely many orbits of $X$ in $\sink(X)$.
Since $f\in \mathcal{R}_{11}$, we have from Theorem \ref{fui} that $H_{i_k}$ is an attractor
which is either hyperbolic or singular-hyperbolic for $X$, $\forall 1\leq k\leq d$.
From this we obtain the result.
\end{proof}

\vspace{5pt}

\begin{proof}[Proof of Theorem B]
Suppose by contradiction that there is
a $C^1$ generic three-dimensional flow of an orientable manifold
such that $\cl(\sink(X))\setminus \sink(X)$ has a LPF-dominated splitting.
Then, $\card(\sink(X))=\infty$ and
$X$ has finitely many periodic sinks with
nonreal eigenvalues. Since $X$ is $C^1$ generic, we obtain
that the number of orbits of sinks with nonreal eigenvalues is locally constant at $X$.
From this we can assume without loss of generality that
every sink of a nearby flow is periodic with real eigenvalues.
Furthermore, we obtain the following alternatives:
If $Ind(\sigma)=2$, then $\sigma$ is Lorenz-like for $X$ and satisfies
$$
(\cl(\sink(X))\setminus \sink(X))\cap W^{ss,X}(\sigma)=\{\sigma\},
$$
and, if $Ind(\sigma)=1$, then $\sigma$ is Lorenz-like for $-X$ and satisfies
$$
(\cl(\sink(X))\setminus \sink(X))\cap W^{uu,X}(\sigma)=\{\sigma\}.
$$
As before, these alternatives imply the following ones:
\begin{enumerate}
\item
If $Ind(\sigma)=2$, then every $\sigma'\in \Sing(X)\cap \cl(W^u(\sigma))$ is Lorenz-like for $X$.
\item
If $Ind(\sigma)=1$, then every $\sigma'\in \Sing(X)\cap \cl(W^s(\sigma))$ is Lorenz-like for $-X$.
\end{enumerate}

For any $p\in\per(X)$ we denote by $\lambda(p,X)$ and $\mu(p,X)$ the two eigenvalues of $p$ so that
$$
|\lambda(p,X)|\leq |\mu(p,X)|.
$$
The corresponding eigenspaces will be denoted by $E^{-,X}_p$ and $E^{+,X}_p$.
We have the symmetric relations
$$
\lambda(p,-X)=\mu^{-1}(p,X),
\mu(p,-X)=\lambda^{-1}(p,X), E^{-,-X}_p=E^{+,X}_p,E^{+,-X}_p=E^{-,X}_p.
$$

We obtain from the fact that the number of sinks with nonreal eigenvalues is locally
constant at $X$ that there is a fixed number $0<\lambda<1$ and a neighborhood $\mathcal{U}_X$ of $X$
satisfying:
\begin{enumerate}
\item[(a)]
$\frac{|\lambda(p,Y)|}{|\mu(p,Y)|}\leq \lambda^{t_{p,Y}}$ and \item[(b)]
$\ang(E^{-,X}_p,E^{+,Y}_p)>\alpha$, for every $(p,Y)\in\sink(Y)\times \mathcal{U}_X$.
\end{enumerate}
Using these properties we obtain as in the proof of Theorem \ref{peo} that
there are neighborhoods $\mathcal{V}_X$ of $X$, $U_\sigma$ of $\sigma$ and
$\beta_\sigma>0$ such that if $Y\in \mathcal{V}_X$ and $x\in \sink(Y)$ satisfies $O_Y(x)\cap U_\sigma\neq\emptyset$, then
$$
\ang(E^{-,Y}_x,E^{cu,Y}_x)>\beta_\sigma,\quad\quad\mbox{ if }Ind(\sigma)=2
$$
and
$$
\ang(E^{-,-Y}_x,E^{cu,-Y}_x)>\beta_\sigma,\quad\quad\mbox{ if }Ind(\sigma)=1.
$$
Consequently there are a neighborhood $\mathcal{K}_X$ of $X$, $0<\rho<1$, $c>0$, $\delta>0$ and $T_0>0$ satisfying
the following properties for every $Y\in\mathcal{K}_X$ and every $x\in \cl(\sink(Y))\setminus \sink(Y)$
satisfying $t_{x,Y}>T_0$ and $O_Y(x)\cap B_\delta(\sigma)\neq\emptyset$:
\begin{itemize}
\item
If $Ind(\sigma)=2$, then
$$
\|DY_T(p)/E^{-,Y}_p\|\cdot \|DY_{-T}(p)/E^{cu,Y}_{Y_{-T}(p)}\|\leq c\rho^T, \quad\quad\forall T>0.
$$
\item
If $Ind(\sigma)=1$, then
$$
\|D(-Y)_T(p)/E^{-,-Y}_p\|\cdot \|D(-Y)_{-T}(p)/E^{cu,-Y}_{(-Y)_{-T}(p)}\|\leq c\rho^T, \quad\quad\forall T>0.
$$
\end{itemize}

Using these dominations as before we
obtain the following:
\begin{itemize}
\item
If $Ind(\sigma)=2$, then $\cl(W^u(\sigma))$ is a singular-hyperbolic attractor for $X$.
\item
If $Ind(\sigma)=1$, then
$\cl(W^s(\sigma))$ is a singular-hyperbolic attractor for $-X$.
\end{itemize}
Since a singular-hyperbolic attractor for either $X$ or $-X$ cannot be accumulated by sinks
we conclude that
$$
\Sing(X)\cap(\cl(\sink(X))\setminus\sink(X))=\emptyset.
$$
Since there is a LPF-dominated splitting, we conclude
that $\cl(\sink(X))\setminus\sink(X)$ is a hyperbolic set.
Since there are only a finite number of orbits of sinks in a neighborhood of a hyperbolic set,
we conclude that
$\card(\sink(X))<\infty$ which is absurd.
This concludes the proof.
\end{proof}

\end{document}